\documentclass[12pt]{amsart}  

\usepackage[latin1]{inputenc}
\usepackage{amsmath} 
\usepackage{amsfonts}
\usepackage{amssymb}
\usepackage{stmaryrd}
\usepackage{latexsym} 
\usepackage{graphicx}
\usepackage{subfigure}
\usepackage{color}
\usepackage{hyperref}
\usepackage{verbatim}
\usepackage[all]{xy}
\usepackage{graphics}
\usepackage{pdfsync}
\usepackage{tikz}
\usepackage{url}
\usepackage{enumerate}

\theoremstyle{plain}
\newtheorem{theorem}{Theorem}
\newtheorem*{theorem*}{Theorem}
\newtheorem{proposition}[theorem]{Proposition}

\newtheorem{lemma}[theorem]{Lemma}
\newtheorem{corollary}[theorem]{Corollary}
\newtheorem{conjecture}[theorem]{Conjecture}
\newtheorem{definition}[theorem]{Definition}
\newtheorem{example}[theorem]{Example}
 
\newtheorem{observation}[theorem]{Observation}

\theoremstyle{remark}
\newtheorem{remark}[theorem]{Remark}

\newcommand{\bQ}{\mathbb{Q}}

\newcommand{\suchthat}{\;|\;}
\newcommand{\spam}{\operatorname{span}}

\newlength\cellsize 
\savebox2{%
\begin{picture}(15,15)
\put(0,0){\line(1,0){15}}
\put(0,0){\line(0,1){15}}
\put(15,0){\line(0,1){15}}
\put(0,15){\line(1,0){15}}
\end{picture}}
\newcommand\cellify[1]{\def\thearg{#1}\def\nothing{}%
\ifx\thearg\nothing
\vrule width0pt height\cellsize depth0pt\else
\hbox to 0pt{\usebox2\hss}\fi%
\vbox to 15\unitlength{
\vss
\hbox to 15\unitlength{\hss$#1$\hss}
\vss}}
\newcommand\tableau[1]{\vtop{\let\\=\cr
\setlength\baselineskip{-16000pt}
\setlength\lineskiplimit{16000pt}
\setlength\lineskip{0pt}
\halign{&\cellify{##}\cr#1\crcr}}}
\savebox3{%
\begin{picture}(15,15)
\put(0,0){\line(1,0){15}}
\put(0,0){\line(0,1){15}}
\put(15,0){\line(0,1){15}}
\put(0,15){\line(1,0){15}}
\end{picture}}
\newcommand\expath[1]{%
\hbox to 0pt{\usebox3\hss}%
\vbox to 15\unitlength{
\vss
\hbox to 15\unitlength{\hss$#1$\hss}
\vss}}
\newcommand\bas[1]{\omit \vbox to \cellsize{ \vss \hbox to \cellsize{\hss$#1$\hss} \vss}}

\begin{document}

\title[Positivity of trees and cut vertices]{Schur and $e$-positivity of trees and cut vertices}

\author{Samantha Dahlberg}
\address{
School of Mathematical and Statistical Sciences,
Arizona State University,
Tempe AZ 85287-1804, USA}
\email{sdahlber@asu.edu}

\author{Adrian She}
\address{
Department of Computer Science,
University of Toronto, 
Toronto ON M5S 2E4, Canada}
\email{ashe@cs.toronto.edu}

\author{Stephanie van Willigenburg}
\address{
 Department of Mathematics,
 University of British Columbia,
 Vancouver BC V6T 1Z2, Canada}
\email{steph@math.ubc.ca}

\thanks{All authors were supported  in part by the National Sciences and Engineering Research Council of Canada.}

\begin{abstract}
We prove that the chromatic symmetric function of any $n$-vertex tree containing a vertex of degree $d\geq \log _2n +1$ is not $e$-positive, that is, not a positive linear combination of elementary symmetric functions. Generalizing this, we also prove that the chromatic symmetric function of any $n$-vertex connected graph containing a cut vertex whose deletion disconnects the graph into $d\geq\log _2n +1$ connected components is not $e$-positive. Furthermore we prove that any $n$-vertex bipartite graph, including all trees, containing a vertex of degree greater than $\lceil \frac{n}{2}\rceil$ is not Schur-positive, namely not a positive linear combination of Schur functions.  {In complete generality, we prove that if an $n$-vertex connected graph has no perfect matching (if $n$ is even) or no almost perfect matching (if $n$ is odd), then it is not $e$-positive. We hence deduce that many graphs containing the claw are not $e$-positive.}
\end{abstract}

\maketitle
\tableofcontents

\section{Introduction}\label{sec:intro}  The generalization of the chromatic polynomial, known as the chromatic symmetric function, was introduced by Stanley in 1995 \cite{Stan95} and has seen a resurgence of interest and activity recently. Much of this has centred around trying to resolve the 1995 conjecture of Stanley \cite[Conjecture 5.1]{Stan95} and its equivalent incarnation \cite[Conjecture 5.5]{StanStem}, which states that if a poset is $(3+1)$-free, then its incomparability graph is a nonnegative linear combination of elementary symmetric functions, that is, $e$-positive. The study of chromatic symmetric function $e$-positivity \cite{ChoHuh, Dahl, lollipop, Foley, FoleyKin, Gash, GebSag, GP, Hamel, MM, HuhNamYoo,  Wolfe}, and related Schur-positivity \cite{Gasharov, Paw, SW, Stanley2}, is also an active area due to connections to the representation theory of the symmetric and general linear group. 

Many partial results regarding chromatic symmetric functions have been obtained such as  when the graph involved  is the path or the cycle \cite{lollipop, Stan95, Wolfe}, when the graph is formed from complete graphs \cite{ChoHuh, GebSag, MM}, or when a graph avoids another \cite{Foley, Gash, Hamel, Tsujie}. These proofs have not always worked directly with the chromatic symmetric function. Instead, sometimes generalizations of the chromatic symmetric function have been employed such as to quasisymmetric functions \cite{ChoHuh, MM, SW} and noncommutative symmetric functions \cite{GebSag}.

Another research avenue that has seen activity is to determine whether two nonisomorphic trees can have the same chromatic symmetric function \cite{Jose2+1, Jose2, HeilJi, Loebl, MMW, Orellana}. The data for up to 29 vertices \cite{HeilJi} shows that two trees $T_1, T_2$ have the same chromatic symmetric function if and only if $T_1$ and $T_2$ are isomorphic. Further evidence towards this includes that for $T_1$ and $T_2$ to have the same chromatic symmetric function they must have the same number of vertices, edges and matchings, and many of these results have been collected together in \cite{Orellana}. 

In this paper, we meld these two avenues and discover criteria on trees and graphs with cut vertices that ensure they are not $e$-positive or not Schur-positive. In particular, we discover a trove of trees that are not $e$-positive, supporting Stanley's observation from 1995 \cite[p 187]{Stan95} that a tree is likely only to be $e$-positive ``by accident''. More precisely, this paper is structured as follows.

In Section~\ref{sec:background} we review the necessary notions before reducing the graphs we need to study to spiders in Subsection~\ref{subsec:redspiders}.  {We also prove the following in Theorem~\ref{the:perfect_matching} and relate it to the claw in Corollary~\ref{cor:clawsandmatching}.}

\begin{theorem*}  {Let $G$ be an $n$-vertex connected graph. If $G$ has no perfect matching (if $n$ is even) or no almost perfect matching (if $n$ is odd), then $G$ is not $e$-positive.}
\end{theorem*}

In Section~\ref{sec:spiders} we study the $e$-positivity of spiders including showing that a spider with at least three legs of odd length is not $e$-positive in Corollary~\ref{cor:matching_cor}. We also show that if the length of each spider leg is less than half the total number of vertices, then the spider is not $e$-positive in Lemma~\ref{lem:short_legs} and generalize this to trees and graphs in Theorem~\ref{the:gen_short_legs}. In Lemma~\ref{lem:induction_lem}, Theorem~\ref{the:induction_short_res_1} and Theorem~\ref{the:induction_short_res_2}, we show that if a spider is not $e$-positive, then we can create infinitely many more spiders from it that are not $e$-positive. Meanwhile Lemmas~\ref{lem:quotient_construction}, ~\ref{lem:quotient_construction_2} and ~\ref{lem:quotient_construction_3} give divisibility criteria on the total number of vertices, which ensure in Proposition~\ref{prop:all_e_positive_spiders} that the spider is not $e$-positive. Applying these results on spiders yields our most general result, the following, given in Theorem~\ref{the:gen_partial_e_thm}.

\begin{theorem*} If $G$ is an $n$-vertex connected graph with a cut vertex whose deletion produces a graph with $d\geq 3$ connected components such that
$$d \geq \log_2 n + 1$$then $G$ is not $e$-positive. 
\end{theorem*}

We show the utility of our results in Example~\ref{ex:gen_short_legs}, where we easily classify when a windmill graph $W^d_n$ for $d\geq1$, $n\geq 1$ is $e$-positive. In Section~\ref{sec:bipartite} we turn our attention to Schur-positivity, proving the following in Theorem~\ref{the:bipartite_s_pos}.

\begin{theorem*}
If $G$ is an $n$-vertex bipartite graph with a vertex of degree greater than $\lceil \frac{n}{2} \rceil$, then $G$ is not Schur-positive.
\end{theorem*}

Finally, in Section~\ref{sec:further} we conclude with two captivating conjectures on the $e$-positivity of trees.

\section{Background}\label{sec:background} In order to describe our results, let us first recall the necessary combinatorics and algebra. We say a \emph{partition} $\lambda = (\lambda _1, \ldots , \lambda _{\ell(\lambda)})$ of $N$, denoted by $\lambda \vdash N$, is a list of positive integers whose \emph{parts} $\lambda _i$ satisfy $\lambda _1 \geq \cdots \geq \lambda _{\ell(\lambda)}>0$ and {$\sum _{i=1} ^{\ell(\lambda)} \lambda _i=N$}. If we have $j$ parts equal to $i$ then we often denote this by $i^j$. Related to every partition $\lambda$ is its \emph{transpose}, $\lambda ^t = (\lambda _1^t, \ldots , \lambda ^t _{\lambda_1})$, which is the partition of $N$ obtained from $\lambda$ by setting
$$\lambda _i^t = \mbox{ number of parts of }\lambda \geq i.$$For example, if $\lambda =(2,2,1)$ then $\lambda ^t = (3,2)$. 

Given a graph $G$ with vertex set $V_G$ and edge set $E_G$, we say that $G$ is an \emph{$n$-vertex graph}, or has \emph{size} $n$, if $|V_G| =n$. We say that a connected graph $G$ contains a \emph{cut vertex} if there exists a vertex $v\in V_G$ such that the deletion of $v$ and its incident edges yields a graph $G'$ with more than one connected component. A \emph{connected partition} $C$ of an $n$-vertex graph $G$ is a partitioning of its vertex set $V_G$ into $\{V_1, \dots, V_k\}$ such that each induced subgraph formed by the vertices in each subset $V_i$ only is a connected graph. The \emph{type} of a connected partition $C$ is the partition of $n$ formed from sorting the sizes of each set $V_i$ in decreasing order.
We say $G$ \emph{has a connected partition} of type $\lambda$ if and only if there exists a connected partition of $G$ of type $\lambda$, and is \emph{missing a connected partition} of type $\lambda$ otherwise.

\begin{example}\label{ex:connpartition} Consider the $n$-vertex star $S_n$ for $n\geq 4$, consisting of a single vertex connected to $n-1$ vertices of degree 1. The star $S_4$ is below. 
$$
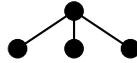
\centering
\begin{tikzpicture}[scale=0.5]
    \coordinate (A) at (0,0);
    \coordinate (B) at (1.5,0);
    \coordinate (C) at (3,0);
    \coordinate (D) at (1.5,1);
    \draw[thick] (A)--(D);
    \draw[thick] (B)--(D);
    \draw[thick] (C)--(D);
    \filldraw (A) circle (7pt);
    \filldraw (B) circle (7pt);
    \filldraw (C) circle (7pt);
    \filldraw (D) circle (7pt);
\end{tikzpicture}$$
The graph $S_n$ has a connected partition of type $\lambda$ if and only if $\lambda = (k, 1^{n-k})$ for some $1 \leq k \leq n$. Examples of connected partitions for $S_4$ of type $(4), (3,1), (2,1^2)$ and $(1^4)$ are below.
$$
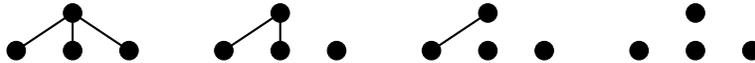
\centering
\begin{tikzpicture}[scale=0.5]
    \coordinate (A) at (0,0);
    \coordinate (B) at (1.5,0);
    \coordinate (C) at (3,0);
    \coordinate (D) at (1.5,1);
    \draw[thick] (A)--(D);
    \draw[thick] (B)--(D);
    \draw[thick] (C)--(D);
    \filldraw (A) circle (7pt);
    \filldraw (B) circle (7pt);
    \filldraw (C) circle (7pt);
    \filldraw (D) circle (7pt);
\end{tikzpicture}
\hspace{1cm}
\begin{tikzpicture}[scale=0.5]
    \coordinate (A) at (0,0);
    \coordinate (B) at (1.5,0);
    \coordinate (C) at (3,0);
    \coordinate (D) at (1.5,1);
    \draw[thick] (A)--(D);
    \draw[thick] (B)--(D);
    \filldraw (A) circle (7pt);
    \filldraw (B) circle (7pt);
    \filldraw (C) circle (7pt);
    \filldraw (D) circle (7pt);
\end{tikzpicture}
\hspace{1cm}
\begin{tikzpicture}[scale=0.5]
    \coordinate (A) at (0,0);
    \coordinate (B) at (1.5,0);
    \coordinate (C) at (3,0);
    \coordinate (D) at (1.5,1);
    \draw[thick] (A)--(D);
    \filldraw (A) circle (7pt);
    \filldraw (B) circle (7pt);
    \filldraw (C) circle (7pt);
    \filldraw (D) circle (7pt);
\end{tikzpicture}
\hspace{1cm}
\begin{tikzpicture}[scale=0.5]
    \coordinate (A) at (0,0);
    \coordinate (B) at (1.5,0);
    \coordinate (C) at (3,0);
    \coordinate (D) at (1.5,1);
    \filldraw (A) circle (7pt);
    \filldraw (B) circle (7pt);
    \filldraw (C) circle (7pt);
    \filldraw (D) circle (7pt);
\end{tikzpicture}
$$
Thus $S_n$ is missing a connected partition of type $(n-2, 2)$ for $n\geq 4$. For example, $S_4$ is missing a connected partition of type $(2,2)$.
\end{example}

\begin{remark}\label{rem:claw} The star $S_4$ is also known as the claw. It is intimately connected to the aforementioned 1995 conjecture of Stanley \cite[Conjecture 5.1]{Stan95} since if a poset is $(3+1)$-free, then its incomparability graph is claw-free. In contrast, the graphs we will study are mostly not claw-free.
\end{remark}

We say that an $n$-vertex graph $G$ has a \emph{perfect matching} if it has a connected partition of type $(2^{\frac{n}{2}})$ and an \emph{almost} perfect matching if it has a connected partition of type $(2^{\frac{n-1}{2}},1)$. Classically stated, a graph $G$ has a perfect matching if there exists a subset of its edges $M\subseteq E_G$, such that every vertex in the graph is incident to exactly one edge in $M$. Similarly $G$ has an almost perfect matching if there exists a vertex $v\in V_G$ whose deletion, along with its incident edges, yields a graph $G'$ that has a perfect matching.

Graphs that will be of particular interest to us will be \emph{trees}, namely connected graphs with no cycles. Recall that degree 1 vertices in trees are called \emph{leaves}, and that a disjoint union of trees is called a \emph{forest}. Two types of tree that will be crucial to our results are paths and spiders. Recall that the \emph{path} $P_n$ of \emph{length} $n$ where $n\geq1$ is the $n$-vertex tree with $n-2$ vertices of degree 2, and 2 leaves, for $n\geq 2$ or a single vertex for $n=1$. Meanwhile, given a partition $\lambda = (\lambda _1, \ldots , \lambda _d) \vdash n-1$ where $d\geq 3$, the \emph{spider} $$S(\lambda) = S(\lambda _1, \ldots , \lambda _d)$$is the $n$-vertex tree consisting of $d$ disjoint paths $P_{\lambda _1}, \ldots , P_{\lambda _d}$ (each respectively called a \emph{leg} of \emph{length} $\lambda _i$ for $1\leq i \leq d$) and a vertex (called the \emph{centre}) joined to a leaf in each path. Extending this notation, $S(i,\lambda)$ is the $(n+i)$-vertex spider with legs of length $i, \lambda _1, \ldots , \lambda _d$.

\begin{example}\label{ex:spider} The $n$-vertex star $S_n$ for $n\geq4$ is also the spider $S(1^{n-1})$. The $7$-vertex spider $S(4,1,1)$ is below.
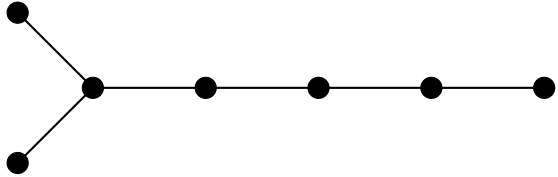
\begin{center}
\begin{tikzpicture}
\filldraw [black] (0,0) circle (4pt);
\filldraw [black] (1,1) circle (4pt);
\filldraw [black] (0,2) circle (4pt);
\filldraw [black] (2.5,1) circle (4pt);
\filldraw [black] (4,1) circle (4pt);
\filldraw [black] (5.5,1) circle (4pt);
\filldraw [black] (7,1) circle (4pt);
\draw[thick] (0,0)--(1,1);
\draw[thick] (0,2)--(1,1);
\draw[thick] (1,1)--(7,1);
\end{tikzpicture}
\end{center}
\end{example}

We now turn to the algebra we will need. The algebra of symmetric functions is a subalgebra of $\bQ [[ x_1, x_2, \ldots ]]$ that can be defined as follows. The \emph{$i$-th elementary symmetric function} $e_i$ for $i\geq 1$ is given by
$$e_i = \sum _{j_1<\cdots < j_i} x_{j_1}\cdots x_{j_i}$$and given a partition $\lambda = (\lambda _1,  \ldots , \lambda _{\ell(\lambda)})$  the \emph{elementary symmetric function} $e_\lambda$ is given by
$$e_\lambda = e_{\lambda _1} \cdots e_{\lambda _{\ell(\lambda)}}.$$The \emph{algebra of symmetric functions}, $\Lambda$, is then the graded algebra
$$\Lambda = \Lambda ^0 \oplus \Lambda ^1 \oplus \cdots$$where $\Lambda ^0 = \spam \{1\} = \bQ$ and for $N\geq 1$
$$\Lambda ^N = \spam \{e_\lambda \suchthat \lambda \vdash N\}.$$Moreover, the elementary symmetric functions form a basis for $\Lambda$. Perhaps the most studied basis of $\Lambda$ is the basis of Schur functions. For a partition $\lambda = (\lambda _1,  \ldots, \lambda _{\ell(\lambda)})$,   the \emph{Schur function} $s_\lambda $ is given by
\begin{equation}\label{eq:JT}
s_{\lambda }=\det \left( e_{\lambda ^t_i -i +j}\right) _{1\leq i,j \leq \lambda_1}
\end{equation}where if $\lambda ^t_{i}-i+j <0$ then $e _{\lambda ^t_{i}-i+j}=0$.

If a symmetric function can be written as a nonnegative linear combination of elementary symmetric functions then we say it is \emph{$e$-positive}, and likewise if a symmetric function can be written as a nonnegative linear combination of Schur functions then we say it is \emph{Schur-positive}. Although not clear from  \eqref{eq:JT}, it is a classical result that any $e$-positive symmetric function is Schur-positive, however Example~\ref{ex:S411} shows that the converse does not hold.

However, the symmetric functions that we will focus on will be the chromatic symmetric function of a graph, which is reliant on a graph that is \emph{finite} and \emph{simple} and we will assume that our graphs satisfy these properties from now on.
Given a graph, $G$, with vertex set $V_G$ a \emph{proper colouring} $\kappa$ of $G$ is a function
$$\kappa : V_G\rightarrow \{1,2,\ldots\}$$such that if $v_1, v_2 \in V_G$ are adjacent, then $\kappa(v_1)\neq \kappa(v_2)$. Then the chromatic symmetric function is defined as follows.

\begin{definition}\cite[Definition 2.1]{Stan95}\label{def:chromsym} For an $n$-vertex graph $G$ with vertex set $V_G=\{v_1, \ldots, v_n\}$, the \emph{chromatic symmetric function} is defined to be
$$X_G = \sum _\kappa x_{\kappa(v_1)}\cdots x_{\kappa(v_n)}$$
where the sum is over all proper colourings $\kappa$ of $G$. \end{definition}

For succinctness, if we say that a graph $G$ is $e$-positive or Schur-positive, then we mean that $X_G$ is $e$-positive or Schur-positive, respectively.

\begin{example}\label{ex:S411} The spider $S(4,1,1)$ from Example~\ref{ex:spider} is not $e$-positive, but is Schur-positive since
\begin{align*}
X_{S(4,1,1)}&=e_{(2^3,1)} + 4e_{(3,2,1^2)} - 3e_{(3,2^2)} + 10e_{(3^2,1)} \\
&+ 10e_{(4,2,1)}+ 17e_{(4, 3)} + 4e_{(5, 1^2)} + 3e_{(5, 2)} + 11e_{(6, 1)} + 7e_{(7)}\\
&=64s_{(1^7)} + 88s_{(2,1^5)} + 76s_{(2^2, 1^3)} + 57s_{(2^3,1)} + 36s_{(3,1^4)} \\
&+ 36s_{(3,2,1^2)} + 18s_{(3,2^2)} + 4s_{(3^2,1)} + 5s_{(4,1^3)} + 6s_{(4,2,1)} + s_{(4,3)}.
\end{align*}
\end{example} 

The following result gives us one way to test whether a graph is $e$-positive, and will be vital in many of our proofs.

\begin{theorem}\cite[Proposition 1.3.3]{Wolfgang} \label{the:e_positivity_crit}
If a connected $n$-vertex graph $G$ is $e$-positive, then $G$ has a connected partition of type $\mu$ for every partition $\mu \vdash n$.
\end{theorem}

Hence, to prove that a graph $G$ is not $e$-positive, it suffices to find a partition $\mu$ such that $G$ is missing a partition of type $\mu$. In particular, we get the following. 

\begin{theorem}\label{the:perfect_matching}  {Let $G$ be an $n$-vertex connected graph. If $G$ has no perfect matching (if $n$ is even) or no almost perfect matching (if $n$ is odd), then $G$ is not $e$-positive.  In particular, let  $T$ be an $n$-vertex tree. If $T$ has no perfect matching (if $n$ is even) or no almost perfect matching (if $n$ is odd), then $T$ is not $e$-positive.}\end{theorem}

 {It is a well-known result that if a connected graph with an even number of vertices is claw-free, then it has a perfect matching. This immediately yields the following corollary to Theorem~\ref{the:perfect_matching}.}

\begin{corollary}\label{cor:clawsandmatching}  {Let $G$ be a connected graph with an even number of vertices and no perfect matching. Then $G$ contains the claw and is not $e$-positive.}
\end{corollary}

\begin{remark}\label{rem:clawsandmatching}  {Note that Corollary~\ref{cor:clawsandmatching} cannot be strengthened further regarding graphs that contain the claw since $S(2,1,1)$ and $S(6,2,1)$ both contain the claw and are $e$-positive. However, $S(2,1,1)$ has an odd number of vertices, and $S(6,2,1)$ has a perfect matching.}
\end{remark}

It is also known when a tree has a perfect matching by the following specialization of Tutte's Theorem on graphs with perfect matchings \cite{Tutte}.

\begin{lemma}\label{lem:perfect_matching} Let $T$ be a tree. Then $T$ has a perfect matching if and only if for every vertex $v$, the deletion of $v$ and its incident edges produces a forest with exactly one connected component with an odd number of vertices.
\end{lemma}

As an example, we use the two theorems above to test the star to see in two ways that it is not $e$-positive.

\begin{example}\label{ex:perfect_matching} The $n$-vertex star $S_n$ for $n\geq4$ is not $e$-positive since it is missing a connected partition of type $(n-2,2)$ by Example~\ref{ex:connpartition}. 
It is also not $e$-positive since it is missing a connected partition of type $(2^{\frac{n}{2}})$ for $n$ even and $(2^{\frac{n-1}{2}},1)$ for $n$ odd. \end{example}

Note, however, that the converse of Theorem~\ref{the:e_positivity_crit} and of Theorem~\ref{the:perfect_matching} is false since the spider $S(4,1,1)$ is not $e$-positive by Example~\ref{ex:S411} and yet has a connected partition of every type.

\subsection{The reduction to spiders}\label{subsec:redspiders} The next two lemmas allow us to reduce our study of connected graphs to the study of spiders. For ease of notation in the proof of the next lemma, in the $i^{th}$ leg of a spider, which has $\lambda_i$ vertices, label the vertices by $\{s_{i,1}, \dots, s_{i,\lambda_i}\}$ where $s_{i,1}$ is the vertex connected to the centre, $s_{i,\lambda_i}$ is a leaf, and there are edges between each $s_{i,j}$ and $s_{i,j+1}$ for $1\leq j \leq \lambda _i -1$. 

\begin{lemma}\label{lem:spider_lem}
Let $T$ be a tree with a vertex of degree $d \geq 3$, and let $v$ be any such vertex. Let $(t_1, \dots, t_d)$ be the partition whose parts denote the sizes of the subtrees $(T_1, \dots, T_d)$ rooted at each vertex adjacent to $v$. If $T$ has a connected partition of type $\mu$, then the spider $S = S(t_1, \dots, t_d)$ has a connected partition of type $\mu$ as well. 
\end{lemma}

\begin{proof}
Let $C = \{V_1, \dots, V_k\}$ be a connected partition of $T$ of type $\mu$. {We will work towards constructing a connected partition of $S$ of type $\mu$.} Without loss of generality, suppose $v \in V_1$. Since $T$ is a tree, no subset in $C$ contains vertices from two different subtrees unless vertex $v$ is also included. Therefore, all subsets in $C$ except possibly $V_1$ contain vertices from one subtree only. Hence, let $V_1$ contains $n_i$ vertices from subtree $T_i$ for $1 \leq i \leq d$, and let $\mathcal{T}_i \subset C$ denote the sets in the partition with vertices in subtree $T_i$ only. Notice that each $\mathcal{T}_i$ contains $t_i - n_i$ vertices. 

A connected partition $\{W_1, \dots, W_k\}$ of the same type $\mu$ in $S$ may now be formed as follows. Let $W_1 = \{v\} \cup \{s_{i,j}: 1 \leq i \leq d, 1 \leq j \leq n_i\}$ where $v$ is the centre of $S$. Note that {$|W_1|=|V_1|$ and $W_1$ is connected.} Now notice that $S$ with deletion of all vertices in $W_1$ and their incident edges is now a collection of $d$ disjoint paths of length $t_i - n_i$ for $1 \leq i \leq d$. For each $\mathcal{T}_i$, let $\nu_i \vdash (t_i - n_i)$ be the partition formed from the size of each set in $\mathcal{T}_i$. Since a path may be decomposed into a connected partition of any type, in particular, a connected partition of type $\nu_i$ can be formed from a path of length $t_i - n_i$. Hence, the result follows.
\end{proof}

In fact, the above argument can be generalized as follows.

\begin{lemma}\label{lem:gen_spider_lem}
Let $G$ be a connected graph with a cut vertex $v$  whose deletion produces a graph with connected components $(C_1, \dots, C_d)$ with $d \geq 3$. Let $(c_1, \dots, c_d)$ be the partition whose parts denote the sizes of each of these connected components. If $G$ has a connected partition of type $\mu$, then the spider $S = S(c_1, \dots, c_d)$ has a connected partition of type $\mu$ as well. 
\end{lemma}

\begin{proof}
Suppose $C$ is a connected partition of $G$ of type $\mu$. By assumption that $v$ is a cut vertex, any path between $w, u$ in distinct connected components $C_i$ and $C_j$ passes through $v$, for otherwise the deletion of $v$ would not leave $C_i$ and $C_j$ as distinct connected components. Hence, there are no edges between any distinct $C_i$ and $C_j$. The proof then proceeds as before in Lemma~\ref{lem:spider_lem} as $C$ contains exactly one set $V_1$ including the vertex $v$ and possibly some other vertices in $(C_1, \dots, C_d)$, and every other set $V_i$ in $C$ contains vertices from exactly one connected component $C_j$.    
\end{proof}

\section{The $e$-positivity of spiders}\label{sec:spiders} We now work towards our first result on spiders by classifying when they have perfect and almost perfect matchings, for which we will need the following straightforward observation.

\begin{observation}\label{obs:matchedpaths} A path, $P_n$, has a perfect matching if and only if $n$ is even.
\end{observation}

With this observation we can now classify when a spider has a perfect or almost perfect matching.

\begin{lemma}\label{lem:spider_matching} We have the following.
\begin{enumerate}[(a)]
\item A spider has a perfect matching if and only if it has exactly one leg of odd length. 
\item A spider has an almost perfect matching if and only if it has zero or two legs of odd length. 
\end{enumerate}
\end{lemma}

\begin{proof}
Suppose a spider $S$ has a perfect matching. The deletion of the centre of $S$ produces a union of disjoint paths, and exactly one of these paths must have odd length by Lemma~\ref{lem:perfect_matching}. Conversely, if $S$ has exactly one leg of odd length, by matching the centre to its neighbour in that leg, a perfect matching exists by Observation~\ref{obs:matchedpaths} since the deletion of these vertices and their incident edges then produces a set of disjoint paths, each of even length.

Now suppose a spider $S$ has an almost perfect matching {so that there exists a vertex $v'$ whose deletion results in a perfect matching}. If the deletion of the centre produces a graph with a perfect matching, all legs in $S$ have even length, by Observation~\ref{obs:matchedpaths}. Otherwise, {$v'$} was in one of the legs $L$. Hence, unless the centre had degree 3 and $v'$ was adjacent to it, the deletion of $v'$ produces a spider $S'$ with an even number of vertices and a path of even length, which may {have 0 vertices}. Notice that an odd number of vertices in total was deleted from $S$ to form $S'$. 

Hence, if the leg $L$ in $S'$ now has odd length, {then} every other leg in $S$ had even length by the first part. Therefore, all legs in $S$ had even length as an odd number of vertices was deleted from $L$ to form $S'$. Otherwise, if the leg $L$ in $S'$ now has even length, then some other leg in $S'$ has odd length by the first part, and so as an odd number of vertices was deleted from $L$ to form $S'$,  $S$ originally had 2 legs of odd length. 

In the case where $v'$ is adjacent to a degree 3 centre, the deletion of $v'$ produces two disjoint paths of even length. Hence, in this case, $S$ had exactly 2 legs of odd length. {This is because one of the paths of even length must consist of 2 legs and the centre, so one of the legs must be of odd length. Furthermore, the other path of even length must be a path of length one less than the remaining third leg, due to the deletion of $v'$, so the third leg must be of odd length.}

For the converse, if $S$ has no legs of odd length, then the deletion of the centre produces a disjoint union of paths with even length, which has a perfect matching by Observation~\ref{obs:matchedpaths}. If $S$ has 2 legs of odd length, the the deletion of the leaf in exactly one of those legs produces a spider with exactly one leg of odd length, which has a perfect matching by the first part.
\end{proof}

\begin{corollary}\label{cor:matching_cor}
Every spider with at least 3 legs of odd length is not $e$-positive.
\end{corollary}

\begin{proof}
This follows immediately from Lemma~\ref{lem:spider_matching} and Theorem~\ref{the:perfect_matching}. In particular, if we have an $n$-vertex spider, then a connected partition of type $(2^{\frac{n}{2}})$ for $n$ even or $(2^{\frac{n-1}{2}},1)$ for $n$ odd, is missing.
\end{proof}

\begin{example} \label{ex:threeoddlegs} All spiders $S(\lambda _1, \lambda _2, \lambda _3, 1,1,1)$ are not $e$-positive. \end{example}

However, it is not only the parity of the legs that can determine the $e$-positivity of a spider, but also the length of the legs. 

\begin{lemma}\label{lem:short_legs} Let $\lambda \vdash (n-1)$ have at least 3 parts and $m$ be the maximum of the parts of $\lambda$. If $m < \lfloor \frac{n}{2} \rfloor$, then the $n$-vertex spider $S(\lambda)$ is missing a connected partition of type $$(n-m-1, m+1)$$and hence is not $e$-positive. 
\end{lemma}

\begin{proof} Consider the types of connected partitions that can be formed by the deletion of one edge from $S(\lambda)$. It is straightforward to see that the only connected partitions of type $(n-i, i)$ can be formed where $1\leq i \leq m$. Since $m < \lfloor \frac{n}{2} \rfloor$ we have that $n-(m+1)\geq m+1$ and so $S(\lambda)$ is missing a connected partition of type
$$(n-(m+1), m+1)=(n-m-1, m+1).$$Hence, $S(\lambda)$ is not $e$-positive by Theorem~\ref{the:e_positivity_crit}.
\end{proof}

\begin{remark}\label{rem:short_legs} If $\lambda$ and $m$ are as in Lemma~\ref{lem:short_legs} and $\mu = (n-m-1, m+1)$, then using the Newton-Girard Identities and \cite[Theorem 2.5]{Stan95} one can compute directly that
$$[e_\mu] X_{S(\lambda)} = -n \mbox{ if } m+1\neq n-m-1$$and
$$[e_\mu] X_{S(\lambda)} = -\frac{n}{2} \mbox{ if } m+1= n-m-1$$where $[e_\mu] X_{S(\lambda)}$ denotes the coefficient of $e_\mu$ in $X_{S(\lambda)}$ when expanded as a linear combination of elementary symmetric functions.
\end{remark}

\begin{example}\label{ex:short_legs} If $\lambda = (2,2,1,1)\vdash 6$ then $m=2$ and $S(2,2,1,1)$, below, is missing a connected partition of type $(4,3)$ so is not $e$-positive by Lemma~\ref{lem:short_legs}. More precisely, it is not $e$-positive since $X_{S(2,2,1,1)}$ contains the term $-7e_{(4,3)}$ by the above remark.
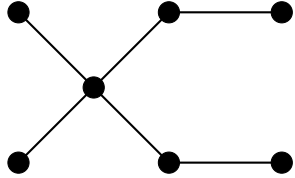
\begin{center}
\begin{tikzpicture}
\filldraw [black] (0,0) circle (4pt);
\filldraw [black] (0,2) circle (4pt);
\filldraw [black] (1,1) circle (4pt);
\filldraw [black] (2,0) circle (4pt);
\filldraw [black] (2,2) circle (4pt);
\filldraw [black] (3.5,0) circle (4pt);
\filldraw [black] (3.5,2) circle (4pt);
\draw[thick] (0,0)--(1,1);
\draw[thick] (0,2)--(1,1);
\draw[thick] (1,1)--(2,0);
\draw[thick] (1,1)--(2,2);
\draw[thick] (2,2)--(3.5,2);
\draw[thick] (2,0)--(3.5,0);
\end{tikzpicture}
\end{center}
\end{example}

For our next three results we construct larger spiders that are not $e$-positive and have been created from smaller ones. Hence, our focus will temporarily not be on the total number of vertices, $n$, but on the number of vertices in an initial set of legs.

\begin{lemma}\label{lem:induction_lem}
Let $i \geq 0$, $\lambda \vdash N$ have at least two parts, and $m$ be the maximum of the parts of $\lambda$. Suppose the spider $S(i,\lambda)$ is missing a connected partition of type $\mu \vdash (i + N+1)$ where $m + 1 \leq \mu_k \leq N$ for each part $\mu_k$ of $\mu$. Then the spider $S(i+N, \lambda)$ is missing a connected partition of type $$(N, \mu)$$ and hence is not $e$-positive.  
\end{lemma}

To aid comprehension, we give an example before the proof.

\begin{example}\label{ex:induction_lem} If $\lambda = (2,1,1)$ then $N=4$ and $m=2$. Setting $i=2$, we know from Example~\ref{ex:short_legs} that $S(2,2,1,1)$ is missing a connected partition of type $\mu = (4,3)\vdash (2+4+1)$ and its parts satisfy $m+1=3\leq 4,3 \leq 4 = N$. Hence, $S(6,2,1,1)$ is missing a connected partition of type $(4,4,3)$.
\end{example}

\begin{proof}
{We will try to construct a connected partition of type $(N, \mu)$ and find this is impossible.}
Suppose $V_1$ is a connected component in a connected partition  of $S = S(i+N,\lambda)$ with $N$ vertices. Let $L'$ denote the vertices that are part of the leg of length $i + N$ in $S$ and {$L_k$} denote the vertices that are part of the leg of length {$\lambda_k$} in $S$.

If $V_1$ does not contain any vertex from $L'$, $V_1$ must contain the centre as each set {$L_k$} by assumption has less than or equal to $N-1$ vertices. As $V_1$ must contain the centre and furthermore cannot contain all $N+1$ vertices in the subtree $S(\lambda)$ in $S$, the deletion of $V_1$ and its incident edges produces an isolated vertex.  Therefore, a connected partition of type $\mu$ cannot be formed from $S$ with all vertices in $V_1$ deleted, as each part of $\mu$ has size greater than or equal to 2 by assumption. 

Hence, $V_1$ must contain some vertices from $L'$. If $V_1$ contained some vertex in some {$L_k$} as well, it must contain the centre since $V_1$ is connected. Once again, since the subtree $S(\lambda)$ has $N+1$ vertices and $V_1$ contains the centre, $V_1$ cannot contain the subtree $S(\lambda)$ entirely and the deletion of $V_1$ from $S$ then produces at least two connected components where one of them is a path with less than or equal to $m$ vertices. Hence, a connected partition of type $\mu$ cannot be formed from the remaining graph as each part has size greater than or equal to $m+1$. 

{Thus, $V_1$ must contain either the centre and vertices from $L'$ only, or vertices from $L'$ only. If $V_1$ contains the centre and vertices from $L'$ only, then, as in the previous paragraph, the deletion of $V_1$ from $S$ then produces at least two connected components where one of them is a path with less than or equal to $m$ vertices. Hence, a connected partition of type $\mu$ cannot be formed from the remaining graph as each part has size greater than or equal to $m+1$.}

Therefore, $V_1$ must contain vertices from $L'$ only, {which} form an induced path of length $N$. Upon deleting $V_1$ and its incident edges, the graph $G_j$ formed is the disjoint union of $S(j,\lambda)$ and a path of length $i - j$ for some $0 \leq j \leq i$.  However, if $G_j$ had a connected partition of type $\mu$ for some $j$, this contradicts the assumption that $S(i,\lambda)$ was missing a connected partition of   type $\mu$, since any $G_j$ can be formed by first deleting an edge from the length $i$ path in $S(i,\lambda)$.

As we have exhausted all possibilities for $V_1$ and $|V_1| = N$, then $S(i+N,\lambda)$ is missing a connected partition of type $(N, \mu)$, and hence it is not $e$-positive by Theorem~\ref{the:e_positivity_crit}. 
\end{proof}

We now come to a substantial way to create families of spiders that are missing a connected partition.

\begin{theorem}\label{the:induction_short_res_1}
Let $\lambda \vdash N$ and $m$ be the maximum of the parts of $\lambda$. If $\max(2m-N+1, 0) \leq i < N$, then the  spider $S(i + Na, \lambda)$ is not $e$-positive for every integer $a \geq 0$.  In particular, the spider $S(i+Na,\lambda)$ is missing a connected partition of type $\mu$ where $$\mu = \begin{cases} (N^{a+1}, i+1) & m \leq i < N \\ (N^a, N+i-m, m+1) & i < m \end{cases}.$$ 
\end{theorem}

\begin{proof} 
Note that if a spider is missing a connected partition then it is not $e$-positive by Theorem~\ref{the:e_positivity_crit}. Hence, we now proceed to find a missing connected partition in all cases.

Before we do this, observe that if $\lambda = (N)$, so $m=N$, then no such $i$ exists. Thus assume that $\lambda \vdash N$ with at least 2 parts. We will now study two cases, $m\leq i$ and $i<m$. In each case we will study the spider $S(i,\lambda)$ before drawing our desired conclusions about $S(i+Na, \lambda)$.

For the first case, if $1\leq m \leq i <N$, then $2m-N+1 \leq m$ since $m\leq N-1$, so our condition on $i$ is trivially satisfied. Also note that all legs in the spider $S(i,\lambda)$ have length less than $\lfloor \frac{i+N+1}{2} \rfloor$. This is because there are $i+N+1$ vertices in total and 
$$\left\lfloor \frac{i + N + 1}{2} \right\rfloor > {\left\lfloor\frac{2i}{2} \right\rfloor}= i \geq m.$$Hence, by Lemma~\ref{lem:short_legs}, since $i$ is the maximum of the parts of $\lambda$ {and $i$}, the spider $S(i,\lambda)$ is missing a connected partition of type $(N, i+1)$. Hence, by Lemma~\ref{lem:induction_lem} since $m+1\leq i+1\leq N$, by assumption, we have that $S(i+N, \lambda)$ is missing a connected partition of type $(N,N,i+1)$.

Consequently, $S(i+Na, \lambda)$ is missing a connected partition of type $(N^{a+1}, i+1)$, since if not then $(a-1)$ connected components with $N$ vertices would have to be contained in the leg of length $i+Na$ (this is because otherwise one of these connected components with $N$ vertices would consist of the centre vertex connected to $N-1$ other vertices, which could not yield a connected partition of type $(N^{a+1}, i+1)$ since $i\geq m$). However, this would imply that $S(i+N, \lambda)$ is not missing a connected partition of type $(N,N,i+1)$, a contradiction.

For the second case, first suppose $0 \leq 2m-N+1 \leq i < m$. Since, therefore, $2m+1\leq i+N$, 
$$\left\lfloor \frac{i+N+1}{2} \right\rfloor \geq \left\lfloor \frac{2m+2}{2} \right\rfloor = m+1 > m > i$$the spider $S(i, \lambda)$ is missing a connected partition of type $(N+i-m, m+1)$
 by Lemma~\ref{lem:short_legs} since $m$ is the length of the longest leg in $S(i, \lambda)$. Hence, by Lemma~\ref{lem:induction_lem}, since $m+1=N+(2m-N+1)-m \leq N+i-m<N$, because $i<m$ by assumption, we have that $S(i+N, \lambda)$ is missing a connected partition of type $(N, N+i-m, m+1)$. 
 
 Alternatively suppose $2m-N+1<0\leq i <m$. The first inequality implies that $m<\frac{N-1}{2}$ so
 {$$\left\lfloor \frac{i+N+1}{2} \right\rfloor  > m$$}and hence the spider $S(i, \lambda)$ is missing a connected partition of type $(N+i-m, m+1)$ by Lemma~\ref{lem:short_legs} since $m$ is the length of the longest leg in $S(i, \lambda)$. Hence, by Lemma~\ref{lem:induction_lem}, since $m+1\leq N+i-m<N$, because $i<m$ by assumption, we have that $S(i+N, \lambda)$ is missing a connected partition of type $(N, N+i-m, m+1)$.
 
Consequently, in both subcases of the second case $S(i+Na, \lambda)$ is missing a connected partition of type $(N^a, N+i-m, m+1)$ since if not then $(a-1)$ connected components with $N$ vertices would have to be contained in the leg of length $i+Na$ (this is because otherwise one of these connected components with $N$ vertices would consist of the centre vertex connected to $N-1$ vertices, which could not yield a connected partition of type $(N,N+i-m,m+1)$). However, this would imply that $S(i+N, \lambda)$ is not missing a connected partition of type $(N,N+i-m,m+1)$, a contradiction.
\end{proof}

\begin{example}\label{ex:induction_short_res_1}
If  $\lambda = (2,1,1)$, then $N=4$ and $m=2$. Since $\max(2m-N+1, 0)=1 \leq 2 < 4=N$ we can set $i=2$ and hence, by Theorem~\ref{the:induction_short_res_1}, every spider $S(2+4a, 2, 1,1)$ is missing a connected partition of type $(4^{a+1}, 3)$ for every integer $a \geq 0$. 
\end{example}

Applying the above theorem now leads us to a surfeit of spiders that are not $e$-positive.

\begin{theorem}\label{the:induction_short_res_2}
Let $\lambda \vdash N$ and $m$ be the maximum of the parts of $\lambda$. If $m < \lfloor \frac{N}{2} \rfloor$ and $i \geq 0$ is an integer, then every spider $S(i,\lambda)$ is not $e$-positive.
\end{theorem}

\begin{proof}
By Theorem~\ref{the:induction_short_res_1}, it suffices to show that if $m < \left \lfloor \frac{N}{2} \right \rfloor$, then $2m-N+1 \leq 0$. If $N$ is even, then $2m < N$, so $2m - N < 0$ and $2m - N + 1 \leq 0$. Otherwise if $N$ is odd, then $2m < N-1$, so once again $2m - N + 1 \leq 0$.  
\end{proof}

An alternative approach for finding a missing connected partition is given in the following lemma. 

\begin{lemma}\label{lem:quotient_construction}
Suppose $S=S(\lambda_1, \dots, \lambda_d)$ is an $n$-vertex spider with $(\lambda_1, \dots, \lambda_d)$ a partition, $\lambda _1 \geq \lambda _2 +\cdots + \lambda _d$, $\lambda_2 \leq \lambda_3 + \dots + \lambda_d$, and with $\lambda_2 \geq 2$. Let $n = q(\lambda_2 + 1) + r$ {where $0\leq r <\lambda _2 +1$}, and $r = qd' + r'$ {where $0\leq r' <q$}. If $\lambda_2 \geq 3$, or if $\lambda_2 = 2$ and $q \geq 3$, then $S$ is missing a connected partition of type $$(\lambda_2 + d' + 2)^{r'} (\lambda_2 + d' + 1)^{q - r'}.$$ 
\end{lemma}

\begin{proof}
Suppose $S$ has a connected partition $C$ of type $(\lambda_2 + d' + 2)^{r'} (\lambda_2 + d' + 1)^{q - r'}$. 
Consider the set $V_1\in C$ containing the leaf on a leg of length $\lambda _2$. Since every set {in $C$} must contain at least $\lambda_2 + 1$ vertices, $V_1$ contains the centre, and hence also all legs of length less than or equal to $\lambda_2$ since all sets in the connected partition have size greater than $\lambda_2$.

Hence, since $\lambda_2 \leq \lambda_3 + \dots + \lambda_d$, $$|V_1|  \geq 1 + \lambda_2 + \dots + \lambda_d \geq 1 + 2 \lambda_2.$$

However, we claim that $1 + 2 \lambda_2 > \lambda_2 + d' + 2$. This is because $$n = 1 + \lambda_1 + \dots + \lambda_d \geq 1 + 2(\lambda_2 + \dots + \lambda_d) > 2(\lambda_2 + 1)$$ by assumption that $\lambda _1 \geq \lambda _2 +\cdots + \lambda _d$. So $q \geq 2$. Hence, if $\lambda_2 \geq 3$, then $$\frac{\lambda_2}{\lambda_2 - 1} = 1 + \frac{1}{\lambda_2 - 1} < 2 \leq q.$$Otherwise if $\lambda_2 = 2$, then $\frac{\lambda_2}{\lambda_2 - 1} = 2 < q$ when $q \geq 3$. Therefore, in either case, 

$$r - r' = qd' \leq r \leq \lambda_2 < q(\lambda_2 - 1),$$which implies that $d' < \lambda_2 - 1$ by dividing both sides by $q$. Adding $\lambda_2 + 2$ to both sides shows that $1 + 2\lambda_2 > \lambda_2 + d' + 2$, which contradicts that $C$ is a connected partition of the desired type as  $V_1$ does not contain $\lambda_2 + d' + 2$ or $\lambda_2 + d' + 1$  vertices. 
\end{proof}

\begin{example}\label{ex:quotient_construction}
Let $S = S(8, 2, 2, 1)$. Since $8\geq 2+2+1$ and $2\leq 2+1$, and $n = 14 = 4(3) + 2$, the conditions of Lemma~\ref{lem:quotient_construction} are met. Since $2 = 4(0) + 2$, $S$ is missing a connected partition of type $(\lambda_2 + 2)^2 (\lambda_2 + 1)^2 = (4,4,3,3)$. 
\end{example}

We now will  generalize Lemma~\ref{lem:quotient_construction} and then bound the number of vertices before giving our key result on the $e$-positivity of a spider.

\begin{lemma}\label{lem:quotient_construction_2}
Let $i \geq 3$ be an integer. Suppose $S(\lambda_1, \dots, \lambda_d)$ is an $n$-vertex spider with  $(\lambda_1, \dots, \lambda_d)$ a partition, with a $\lambda_i \geq 2$ satisfying $\lambda_i \leq \lambda_{i+1} + \dots + \lambda_d$ {where $i<d$}, and $\lambda_j > \lambda_{j+1} + \dots + \lambda_d$ for all $j < i$. Let $n = q(\lambda_i + 1) + r$ {where $0\leq r <\lambda _i +1$}, and $r = qd' + r'$ {where $0\leq r' <q$}. Then $S$ is missing a connected partition of type $$(\lambda_i + d' + 2)^{r'} (\lambda_i + d' + 1)^{q - r'}.$$
\end{lemma}

\begin{proof}
Suppose $S$ has a connected partition $C$ of the desired type.
Consider the set $V_1 \in C$ containing  the leaf on a leg of length $\lambda_i$. Since every set {in $C$} must contain at least $\lambda_i + 1$ vertices, $V_1$ contains the centre, and hence also all legs of length less than or equal to $\lambda_i$ since all sets in the connected partition have size greater than $\lambda_i$. Hence, since $\lambda_i \leq \lambda_{i+1} + \dots + \lambda_d$,

$$|V_1| \geq 1 + \lambda_i + \dots + \lambda_d \geq 1 + 2 \lambda_i.$$

However, we claim that $\frac{\lambda_i}{\lambda_i - 1} \leq 2 < q$. This is because $\lambda_i \geq 2$ so $\frac{\lambda_i}{\lambda_i - 1} \leq 2$,  and repeatedly using the condition that $\lambda_j > \lambda_{j+1} + \dots + \lambda_d$, for all $j < i$ gives

{\begin{align*}n&=1+\lambda _1 + \cdots + \lambda _d\\
&>1+2(\lambda _2 + \cdots + \lambda _d)\\
&>1+4(\lambda _3 + \cdots + \lambda _d)\\
&\vdots\\
&> 1 + 2^{i-1} (\lambda_i + \dots + \lambda_d) \geq 2^{i-1} (\lambda_i + 1).\end{align*}Hence, 
$$(q+1)(\lambda _i +1) > q(\lambda_i + 1) + r = n > 2^{i-1} (\lambda_i + 1)$$and so} $q \geq 2^{i-1} \geq 4>2$ since $i \geq 3$. This implies that $1 + 2 \lambda_i > \lambda_i + d' + 2$ as in the proof of Lemma~\ref{lem:quotient_construction}, which contradicts that $C$ is a connected partition of the desired type as  $V_1$ does not contain {$\lambda_i + d' + 2$ or $\lambda_i + d' + 1$}  vertices. 
\end{proof}

\begin{lemma}\label{lem:count_vertices}
Let $S(\lambda_1, \dots, \lambda_d)$ be an $n$-vertex spider, with  $(\lambda_1, \dots, \lambda_d)$ a partition.
\begin{enumerate}[(a)]
\item If $\lambda_1 \geq \lambda_2 + \dots + \lambda_d$ and $\lambda_i > \lambda_{i+1} + \dots + \lambda_d$ for $2 \leq i \leq d - 1$, then $n > 2^{d-1}$.
\item If $\lambda_1 \geq \lambda_2 + \dots + \lambda_d$, $\lambda_i > \lambda_{i+1} + \dots + \lambda_d$ for {$d\geq 4$ and }$2 \leq i \leq d-2$, and $\lambda_{d-1} = \lambda_d = 1$, then $n > 2^{d-1}$.  
\end{enumerate} 
\end{lemma} 

\begin{proof}
For (a), if the partition $(\lambda_1, \dots, \lambda_d)$ satisfies the conditions stated, then we claim that $\lambda_{d-i} > 2^{i-1} \lambda_d$ for $1 \leq i \leq d-1$. This is true for $i = 1$ by assumption. Next, assuming by induction that $\lambda_{d-j} > 2^{j-1} \lambda_d$ for all $1 \leq j < i$, 

$$\lambda_{d-i} \geq \sum_{j = 0}^{i-1} \lambda_{d-j} > \lambda_d  + \sum_{j=1}^{i-1} 2^{j-1} \lambda_d = \lambda_d + (2^{i-1} - 1) \lambda_d = 2^{i-1} \lambda_d.$$

Hence, the total number of vertices $n$ satisfies 

$$n = \lambda_1 + \dots + \lambda_d + 1 > \lambda_d + \sum_{i=1}^{d-1} 2^{i-1} \lambda_d = {\lambda_d + (2^{d-1} - 1) \lambda_d}  = 2^{d-1} \lambda_d.$$

Since $\lambda_d \geq 1$, then {$n > 2^{d-1} \lambda_d\geq2^{d-1}$.}

For (b), we claim that $\lambda_{d-i} \geq 3 (2^{i-2})$ for $2 \leq i \leq d-2$. This is true for $i=2$ since $\lambda_{d-2} \geq 3$ from the conditions given. Next, assuming by induction that this holds for all $2 \leq j < i$, 

$$\lambda_{d-i} \geq \sum_{j=2}^{i-1} \lambda_{d-j} + \lambda_{d-1} + \lambda_d + 1 \geq 3 + \sum_{j=2}^{i-1} 3 (2^{j-2})  = 3 + 3 (2^{i-2} - 1) = 3 (2^{i-2}).$$

Thus, $\lambda_1 \geq \lambda_2 + \dots + \lambda_d \geq \sum_{i=2}^{d-2} 3  (2^{i-2}) + 2 = 3 (2^{d-3}) - 1.$ 

Hence, the total number of vertices $n$, by considering the centre and lengths of each leg,  satisfies

$$n={1+\lambda_1+\cdots + \lambda _d} \geq \sum_{j=0}^{d-3} 3 (2^j) - 1 + 3 = 3 (2^{d-2} - 1) + 2 = 3 (2^{d-2}) - 1 > 2^{d-1}$$
since $d \geq 3$ for a spider.  
\end{proof}

We now arrive at our most general result on the $e$-positivity of spiders.

\begin{theorem}\label{the:partial_e_thm}
Suppose $S=S(\lambda_1, \dots, \lambda_d)$ is an $n$-vertex spider with a vertex of degree $$d \geq \log_2 n + 1.$$ Then $S$ is missing a connected partition of some type $\mu$, and hence is not $e$-positive.
\end{theorem}

\begin{proof}
By Lemma~\ref{lem:count_vertices}(a), either $\lambda_1 < \lambda_2 + \dots + \lambda_d$ or $\lambda_i \leq \lambda_{i+1} + \dots + \lambda_d$ for some $i \geq 2$ {and $i<d$}, since if not, $d < \log_2 n + 1$.  

{For the first case, if} $\lambda_1 < \lambda_2 + \dots + \lambda_d$, then Lemma~\ref{lem:short_legs} provides a missing connected partition. {For the second case}, if $\lambda_1 \geq \lambda_2 + \dots + \lambda_d$ but $\lambda_2 \leq \lambda_3 + \dots + \lambda_d$ and $\lambda _2\geq 3$,  then Lemma~\ref{lem:quotient_construction} provides a missing connected partition.
If $\lambda _2 =2$ then since $d\geq3$, $\lambda _3$ exists and $\lambda _3 \leq 2$. Hence, $n\geq 9$ so the conditions of Lemma~\ref{lem:quotient_construction} are satisfied and it provides a missing connected partition. Meanwhile, if $\lambda _2=1$ then since $d\geq 3$, $\lambda _3$ exists and {$\lambda _3 = 1$.} Hence, $n\geq 5$ and $d\geq 4$ since $d \geq \log_2 n + 1$ {so $\lambda _4$ exists and $\lambda _4 = 1$.} Thus by Lemma~\ref{lem:spider_matching} the spider does not have a perfect {or almost perfect} matching and hence is missing a connected partition as well.

Otherwise, if the spider does not fall into the {second case subcases} above, {then} let $i \geq 3$ be the least index for which $\lambda_i \leq \lambda_{i+1} + \dots + \lambda_d$ {and $i<d$ is true. Then $d\geq 4$.} If $\lambda_i \geq 2$, Lemma~\ref{lem:quotient_construction_2} provides a missing connected partition. Otherwise, $\lambda_i = 1$. By Lemma~\ref{lem:count_vertices}(b), $i < d-1$ since if $i = d-1$, then $d < \log_2 n + 1$. Hence, the centre of the spider is attached to at least three leaves, and by Lemma~\ref{lem:spider_matching}, the spider does not have a perfect or almost perfect matching, and hence is missing a connected partition. 
\end{proof}

We are now left with spiders $S(\lambda _1, \ldots , \lambda _d)$ with $(\lambda _1, \ldots , \lambda _d)$ a partition,   $\lambda_1 \geq \lambda_2 + \dots + \lambda_d$, and either every other leg $\lambda_i$ satisfies $\lambda_i > \lambda_{i+1} + \dots + \lambda_d$, or $\lambda _{d-1}=\lambda _d = 1$ and all other $\lambda _i$ satisfy this inequality. In this case, a spider may not be $e$-positive but still have a connected partition of every type. The spider $S(6,4,1,1)$ is one such example.
 Otherwise, we can get a missing partition for ``sufficiently large'' spiders by a method similar to Lemma~\ref{lem:quotient_construction_2}.

\begin{lemma}\label{lem:quotient_construction_3}
Suppose $S=S(\lambda_1, \dots, \lambda_d)$ is an $n$-vertex spider with $(\lambda_1, \dots, \lambda_d)$ a partition. Pick some  $\lambda_i$ with $i \geq 2$ and let $n = q (\lambda_i + 1) + r$ {where $0\leq r <\lambda _i +1$}, $r = qd' + r'$ {where $0\leq r' <q$}, and $t = \lambda_{i+1} + \dots + \lambda_d$ {where $i<d$ and $t>1$}. If $q \geq \frac{\lambda_i + 1}{t - 1}$, then $S$ is missing a connected partition of type $$(\lambda_i + d' + 2)^{r'} (\lambda_i + d' + 1)^{q-r'}.$$ 
\end{lemma}

\begin{proof}
Suppose $S$ has a connected partition $C$ of the desired type.  Consider the set $V_1\in C$ containing the leaf on a leg of length $\lambda _i$. Since every set {in $C$} must contain at least $\lambda_i + 1$ vertices, $V_1$ contains the centre, and hence also all legs of length less than or equal to $\lambda_i$ since all sets in the connected partition have size greater than $\lambda_i$. Hence, $|V_1| \geq 1 + \lambda_i + \dots + \lambda_d = 1 + \lambda_i + t$, but we claim that this is greater than $\lambda_i + d' + 2$. It suffices to show that $d' < t - 1$. This follows since,

$$d' = \frac{r-r'}{q} \leq \frac{\lambda_i}{q} \leq \frac{\lambda_i}{\lambda_i + 1} (t-1) < t-1$$by the assumption on the value of the quotient $q$.   Hence, $1+\lambda _i+t>\lambda _i+d'+2$, which contradicts that $C$ is a connected partition of the desired type as $V_1$ does not contain $\lambda_i + d' + 2$ or $\lambda_i + d' + 1$ vertices.
\end{proof}

Hence, if  $\lambda_{i+1}, \dots, \lambda_d$ are fixed, then for sufficiently large values of the sum $\lambda_1 + \dots + \lambda_i$, the spider $S(\lambda_1, \dots, \lambda_d)$ will be missing a connected partition of some type. 

\begin{example}\label{ex:136411} If $\lambda = (13, 6,4,1,1)$ and let $i=2$ so that $\lambda _i = 6$, then $26 = 3(7)+5$, $5=3(1) +2$ and $6=4+1+1$. Hence, $q=3$, $\lambda _i +1 = 7$, and $t=6$. Since $3\geq \frac{7}{5}$, we have that $S(\lambda)$ is missing a connected partition of type
$$(6+1+2)^2(6+1+1)=(9,9,8).$$
\end{example}

To end this section, we collect together the various latter lemmas on missing partitions of various types and draw the following conclusion on $e$-positivity, which is immediate by Theorem~\ref{the:e_positivity_crit}.

\begin{proposition}\label{prop:all_e_positive_spiders} If a spider satisfies the criteria of Lemmas~\ref{lem:quotient_construction}, ~\ref{lem:quotient_construction_2} or ~\ref{lem:quotient_construction_3}, then it is not $e$-positive.
\end{proposition}

\section{The $e$-positivity of trees and cut vertices}\label{sec:trees} We can now use our results from the previous section in conjunction with Lemmas~\ref{lem:spider_lem} and ~\ref{lem:gen_spider_lem} to deduce criteria for $e$-positivity of trees and graphs in general.

\begin{theorem}\label{the:gen_partial_e_thm} If $G$ is an $n$-vertex connected graph with a cut vertex whose deletion produces a graph with $d\geq 3$ connected components such that
$$d \geq \log_2 n + 1$$then $G$ is not $e$-positive. 

In particular, if $T$ is an $n$-vertex tree with a vertex of degree $d\geq 3$ such that
$$d \geq \log_2 n + 1$$then $T$ is not $e$-positive.
\end{theorem}

\begin{proof} For the first part, by Lemma~\ref{lem:gen_spider_lem} and Theorem~\ref{the:partial_e_thm} every such $n$-vertex graph is missing a connected partition of some type, and hence is not $e$-positive by Theorem~\ref{the:e_positivity_crit}. For the second part we can either repeat the above argument but this time using Lemma~\ref{lem:spider_lem} instead of Lemma~\ref{lem:gen_spider_lem}, or we can note that if we delete a vertex of degree $d$ from a tree, then $d$ connected components remain.
\end{proof}

As a simple example, every tree with 1000 vertices that contains a vertex of degree 11 or more is not $e$-positive. 
In fact, we can use every result from the previous section on spiders that involves a missing connected partition to obtain a result on trees, cut vertices and $e$-positivity. We illustrate this using Lemma~\ref{lem:short_legs}.

\begin{theorem}\label{the:gen_short_legs} If $G$ is an $n$-vertex connected graph with a cut vertex whose deletion produces a graph with connected components $C_1, \ldots , C_d$ such that $d\geq 3$ and $|V_{C_i}|< \lfloor \frac{n}{2}\rfloor$ for all $1\leq i \leq d$, then $G$ is not $e$-positive. 

In particular, if $T$ is an $n$-vertex tree with a vertex for degree $d\geq3$ whose deletion produces subtrees $T_1, \ldots , T_d$  and $|V_{T_i}|< \lfloor \frac{n}{2}\rfloor$ for all $1\leq i \leq d$, then $T$ is not $e$-positive. 
\end{theorem}

\begin{proof} For the first part, by Lemma~\ref{lem:gen_spider_lem} and Lemma~\ref{lem:short_legs} every such $n$-vertex graph is missing a connected partition type $(n-m-1, m+1)$ where
$$m=\max \{|V_{C_1}|, \ldots, |V_{C_d}|\}$$and hence is not $e$-positive by Theorem~\ref{the:e_positivity_crit}. For the second part we can either repeat the above argument but this time using Lemma~\ref{lem:spider_lem} instead of Lemma~\ref{lem:gen_spider_lem}, or we can note that if we delete a vertex of degree $d$ from a tree, then $d$ connected components remain.
\end{proof}

As a more meaningful example, we will now classify when a windmill graph is $e$-positive.

\begin{example}\label{ex:gen_short_legs} Let $K_n$ be the \emph{complete graph} on $n$-vertices, namely the $n$-vertex graph in which every two vertices are adjacent. Let $W^d_n$ for $d\geq 1, n\geq 1$ be the \emph{windmill graph} in which $d$ copies of $K_n$ all have one common vertex $c$. For example, $W^4_3$ is below.

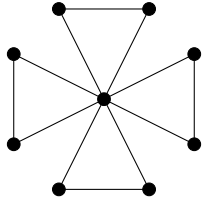
\begin{center}
\begin{tikzpicture}[scale=0.6]
\filldraw [black] (0,0) circle (4pt);
\filldraw [black] (-2,1) circle (4pt);
\filldraw [black] (-2,-1) circle (4pt);
\filldraw [black] (2,1) circle (4pt);
\filldraw [black] (2,-1) circle (4pt);
\filldraw [black] (1,2) circle (4pt);
\filldraw [black] (-1,2) circle (4pt);
\filldraw [black] (1,-2) circle (4pt);
\filldraw [black] (-1,-2) circle (4pt);
\draw (0,0)--(-2,1)--(-2,-1)--(0,0);
\draw (0,0)--(2,1)--(2,-1)--(0,0);
\draw (0,0)--(1,2)--(-1,2)--(0,0);
\draw (0,0)--(1,-2)--(-1,-2)--(0,0);
\end{tikzpicture}
\end{center}

Note that $W^d_n$ has $d(n-1)+1$ vertices. Also note that {for $n>1$} the deletion of $c$ produces $d$ connected components, more precisely $d$ copies of  $K_{n-1}$ each with $n-1$ vertices. Hence, for $d\geq 3$ since 
$$(n-1)<\left\lfloor \frac{d(n-1)+1}{2}\right\rfloor$$every $W^d_n$ for $d\geq 3, {n>1}$ is not $e$-positive by Theorem~\ref{the:gen_short_legs}. In contrast, by say \cite[Theorem 8]{ChovW} and \cite[Corollary 3.6]{Stan95} respectively, every $W^1_n=K_n$ for {$n>1$} and $W^2_n$ for {$n>1$} is $e$-positive. {Lastly, note that $W^d_1$ for $d\geq 1$ is $K_1$ and so, by say \cite[Theorem 8]{ChovW},  is $e$-positive.}
\end{example}

\section{The Schur-positivity of bipartite graphs}\label{sec:bipartite} While $e$-positivity implies Schur-positivity, it is possible for a graph to be Schur-positive but not $e$-positive, for example the spider $S(4,1,1)$ from Examples~\ref{ex:spider} and ~\ref{ex:S411}. 

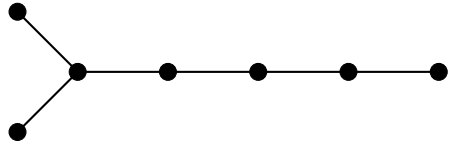
\begin{center}
\begin{tikzpicture}[scale=0.8]
\filldraw [black] (0,0) circle (4pt);
\filldraw [black] (1,1) circle (4pt);
\filldraw [black] (0,2) circle (4pt);
\filldraw [black] (2.5,1) circle (4pt);
\filldraw [black] (4,1) circle (4pt);
\filldraw [black] (5.5,1) circle (4pt);
\filldraw [black] (7,1) circle (4pt);
\draw[thick] (0,0)--(1,1);
\draw[thick] (0,2)--(1,1);
\draw[thick] (1,1)--(7,1);
\end{tikzpicture}
\end{center}

Again we can determine whether trees or certain graphs are not Schur-positive using a vertex degree criterion, but this time we will need the dominance order on partitions, bipartite graphs, and stable partitions.

For the first of these, recall that given two partitions of $N$, $\lambda = (\lambda _1, \ldots , \lambda _{\ell(\lambda)})$ and $\mu = (\mu _1, \ldots , \mu _{\ell(\mu)})$, we  say that $\lambda$ \emph{dominates} $\mu$, denoted by $\lambda \geq _{dom} \mu$ if
$$\lambda _1 + \cdots + \lambda _i \geq \mu _1 + \cdots + \mu _i $$for all $1\leq i \leq \min\{\ell(\lambda), \ell(\mu)\}$. For the second of these, recall that a graph $G$ is \emph{bipartite} if there exists a proper colouring of $G$ with 2 colours. For the third of these, we say a \emph{stable} partition of an $n$-vertex graph $G$ is a partitioning of its vertex set $V$ into sets $\{V_1, \dots, V_k\}$ such that every set $V_i$ in the partitioning is an independent set, namely no edge $e\in E_G$ exists between any $v_1, v_2 \in V_i$. 
The \emph{type} of a stable partition is the partition of $n$ formed by sorting the sizes of each set $V_i$ in decreasing order. We say $G$ \emph{has a stable partition} of type $\lambda$ if and only if there exists a stable partition of $G$ of type $\lambda$, and is \emph{missing a stable partition} of type $\lambda$ otherwise.

\begin{example}\label{ex:stablepartition} Consider again the $n$-vertex star $S_n$ for $n\geq 4$. The graph $S_n$ has a stable partition of type $(n-1, 1)$ but is missing a stable partition of type $(n-2,2)$.
\end{example}

Stable partitions are intimately related to the Schur-positivity of a graph via the following result.

\begin{theorem}\cite[Proposition 1.5]{Stanley2} \label{thm:s_positivity_crit}
Suppose an $n$-vertex graph $G$ has a stable partition of type $\lambda \vdash n$. If $G$ is Schur-positive, then $G$ has a stable partition of type $\mu$ for every $\mu \leq _{dom}\lambda$.
\end{theorem}

This in turn yields a criterion for when a graph is not Schur-positive that is dependent on vertex degrees.

\begin{theorem}\label{the:bipartite_s_pos}
If $G$ is an $n$-vertex bipartite graph with a vertex of degree greater than $\lceil \frac{n}{2} \rceil$, then $G$ is not Schur-positive.

In particular, if $T$ is an $n$-vertex tree with a vertex of degree greater than $\lceil \frac{n}{2} \rceil$, then $T$ is not Schur-positive.
\end{theorem}

\begin{proof}
For the first part, let $G$ be an $n$-vertex bipartite graph with a vertex $v$ of degree $d > \lceil \frac{n}{2} \rceil$. By assumption that $G$ is bipartite, there is a proper colouring of $G$ with two colours. Call these colours red and black, and note that $(V_1, V_2)$ where $V_1$ is the set of the vertices coloured red and $V_2$ is the set of vertices coloured black is a stable partition of $V_G$. The type of this stable partition will be a  partition $(m, n - m)$ where $m > \lceil \frac{n}{2} \rceil$, since the $d > \lceil \frac{n}{2} \rceil$ vertices adjacent to $v$ must be assigned a different colour from the colour of $v$ by assumption.

We claim now that $G$ does not have a stable partition of type $(\lceil \frac{n}{2} \rceil, \lfloor \frac{n}{2} \rfloor)$. Suppose $G$ did have such a partitioning of its vertices into $(V_1, V_2)$ with $|V_1| = \lceil \frac{n}{2} \rceil$ and $|V_2| = \lfloor \frac{n}{2} \rfloor$. If $v \in V_1$, then its neighbours must be in $V_2$, which is impossible since $v$ has degree $d > \lceil \frac{n}{2} \rceil \geq \lfloor \frac{n}{2} \rfloor$. Similarly, if $v \in V_2$, then its neighbours must be in $V_1$, which is impossible since $v$ has degree $d > \lceil \frac{n}{2} \rceil$. 

Since $(\lceil \frac{n}{2} \rceil, \lfloor \frac{n}{2} \rfloor) < _{dom}(m, n - m)$, and $G$ has a stable partition of type $(m, n-m)$, but $G$ is missing a stable partition of type $(\lceil \frac{n}{2} \rceil, \lfloor \frac{n}{2} \rfloor)$, then by Theorem~\ref{thm:s_positivity_crit}, $G$ is not Schur-positive.

For the second part, note that all trees are bipartite.
\end{proof}

\begin{example}\label{ex:bipartite_s_pos} Since the star $S_n$ for $n\geq 4$ is a tree and has one vertex of degree $(n-1)>\lceil \frac{n}{2} \rceil$, it is not Schur-positive, and hence again not $e$-positive, since $e$-positivity implies Schur-positivity.
\end{example}

\section{Further avenues}\label{sec:further} A natural avenue to pursue is to tighten the bound in Theorem~\ref{the:partial_e_thm}, and to this end we conjecture the following, which has been checked for all trees with up to 12 vertices.

\begin{conjecture}\label{con:4vertices} If $T$ is an $n$-vertex tree with a vertex of degree $d\geq 4$, then $T$ is not $e$-positive.
\end{conjecture}

Towards this one can use our techniques to check that certain families of spiders are missing a particular type of partition. However, so far these results have been local to the family of spiders being studied, and there does not seem to exist a natural global type of partition that is missing. Moreover,  as noted just before Lemma~\ref{lem:quotient_construction_3}, there exist spiders that may not be $e$-positive but still have a connected partition of every type, such as $S(6,4,1,1)$. 

In this case we can prove that the family of spiders $S(r,s,1,1)$ is not $e$-positive by first noting that if $r$ or $s$ is odd then $S(r,s,1,1)$  is not $e$-positive by Corollary~\ref{cor:matching_cor}. If $r$ and $s$ are even then we can prove this by using the triple-deletion rule of Orellana-Scott \cite[Theorem 3.1]{Orellana}, and generalized to $k$-deletion by the first and third authors \cite[Proposition 5]{lollipop}, to express $X_{S(r,s,1,1)}$ as a linear combination of chromatic symmetric functions of unions of paths. From here, by using the formula of Wolfe \cite[Theorem 3.2]{Wolfe} for expressing $X_{P_n}$ as a linear combination of elementary symmetric functions we can  show that if $r=2k$ and $s=2\ell$ then
$$[e_{(3,2^{k+\ell})}]X_{S(r,s,1,1)} = -2(r+s)+7,$$which is negative when $r\geq 2$ and $s\geq 2$. This technique can also be used to show that $S(r, 1, 1)$ for $r\geq 3$ has
$$[e_{(r-1,2^2)}]X_{S(r,1,1)} = -(r-1).$$For example, returning to Example~\ref{ex:S411}, note the term $-3e_{(3,2^2)}$ in $X_{S(4,1,1)}$. Direct calculation yields that $S(1,1,1)$ is not $e$-positive but $S(2,1,1)$ is $e$-positive, and hence deducing that $S(r, 1, 1)$ for $r\geq 3$ is not $e$-positive supports Stanley's statement \cite[p 187]{Stan95} that $S(2,1,1)$ is $e$-positive ``by accident''. {Meanwhile, regarding Schur-positivity, we believe that the bound in Theorem~\ref{the:bipartite_s_pos} cannot be improved. This is implied by the following conjecture, which has been checked for all trees with up to 19 vertices, and with which we end.}

{\begin{conjecture}\label{con:Spositivetrees} For all $n\geq 2$, there exists an $n$-vertex tree with a vertex of degree  {$\lfloor \frac{n}{2} \rfloor$} that is  Schur-positive.
\end{conjecture}}
 
\section*{Acknowledgements}\label{sec:acknow} The authors would like to thank John Shareshian for fruitful conversations,  {and the referee for drawing our attention to the connection between claw-free graphs and perfect matchings that then gives Corollary~\ref{cor:clawsandmatching}.}
\bibliographystyle{plain}

\def\cprime{$'$}

\end{document}